\documentclass[12pt, a4paper]{article}

\usepackage{amsfonts}
\usepackage{amsmath}
\usepackage{amssymb}
\usepackage[english]{babel}
\usepackage{amsthm}
\usepackage{graphicx}

\newtheorem {theorem}{Theorem}[section]

\newtheorem {proposition}[theorem]{Proposition}

\theoremstyle{definition}
\newtheorem* {remark}{Remark}

\DeclareMathOperator{\conv}{conv}
\DeclareMathOperator{\aff}{aff}
\DeclareMathOperator{\relint}{rel\, int}

\title{On Delaunay's classification theorem on faces of parallelohedra of codimension three}

\author{Alexander Magazinov \thanks{Supported by the Russian government project 11.G34.31.0053 and RFBR grant 11-01-00633-a.}}

\begin{document}

\maketitle

\begin{abstract}

In 1929 B.~N.~Delaunay proved that there are exactly 5 types of coincidence of parallelohedra at faces of codimension 3.
We give a combinatorial proof of this theorem and prove several additional statements on three-codimensional faces of
parallelohedral tiling.

\end{abstract}

\section{Introduction}

Let $P$ be a convex polytope in $\mathbb R^d$ such that there exists a face-to-face tiling $T(P)$ of $\mathbb R^d$ by translates of $P$.
Then $P$ is called a {\it parallelohedron}. In the present paper $d$ will always stand for the dimension of $P$.

As one can easily see, from the face-to-face property follows that
$$\Lambda(P) = \{\mathbf t: P+\mathbf t \in T(P) \} \quad \text{is a lattice.}$$

In 1897 H.~Minkowski \cite{min1897} established three necessary conditions for a convex polytope $P$ to be a parallelohedron. The statement of the first two is clear.
\begin{enumerate}
	\item A parallelohedron $P$ is a centrally symmetric polytope.
	\item All hyperfaces of $P$ are centrally symmetric. 
\end{enumerate}
The third condition uses the notion of {\it belts}, defined for convex polytopes with centrally symmetric hyperfaces.

Let $Q$ be a convex $d$-dimensional polytope with all hyperfaces being centrally symmetric. Let $F$ be an arbitrary $(d-2)$-dimensional face of $Q$.
Then $F$ determines a {\it belt} of $Q$, which is the set of all hyperfaces of $Q$ parallel to $F$. One can notice that each hyperface of the belt contains
exactly two $(d-2)$-dimensional faces parallel to $F$ and each $(d-2)$-dimensional face parallel to $F$ is shared by exactly two hyperfaces of the belt.

The third Minkowski's condition is as follows.
\begin{enumerate}
	\item[3.] Every belt of $P$ consists of 4 or of 6 hyperfaces.
\end{enumerate}

Later on, in 1954, B.~A.~Venkov \cite{ven1954} has shown that conitions 1 -- 3 are sufficient for a convex polytope $P$ to be a parallelohedron. 
Thus the conditions 1 -- 3 are commonly called {\it Minkowski--Venkov conditions}. 

Let $F$ be a face of $T(P)$. Define the {\it associated cell} $\mathcal D(F)$ as the set of all centers of parallelohedra that share $F$. This notion has been
used, for example, by A.~Ordine \cite{ord2005}.

The family of all associated cells of the tiling has a structure of a cell complex. In this complex $\dim \mathcal D(F) = d - \dim F$.

Now let $\dim F = d-k$. Consider a $k$-dimensional plane $L$ that intersects $F$ transversally. In a small neighborhood of $F$ the section of $T(P)$ by $L$ coincides with a complete $k$-dimensional polyhedral fan, which is called the {\it fan of a face $F$}. (Compare with the definition of a {\it star of
$F$} by Ryshkov and Rybnikov \cite{rry1997}.)

The combinatorics of the fan of $F$ is completely described by the poset of all dual subcells of $\mathcal D(F)$, where the partial order relationship is the inclusion.

In 1929 B.~N.~Delaunay proved the celebrated theorem, which can be stated in the terms of the present paper as follows.

\begin{theorem}\label{delaunay}

Every associated cell of a $(d-3)$-dimensional face of a parallelohedral tiling is combinatorially equivalent to one of the following five 3-dimensional polytopes:
tetrahedron, octahedron, quadrangular pyramid, triangular prism, parallelepiped. Equivalently, every fan of a $(d-3)$-dimensional face is combinatorially 
equivalent to one of the fans shown in Figure 1.

\end{theorem}

\begin{figure}[ht]   
      \centerline{\includegraphics[width = \textwidth]{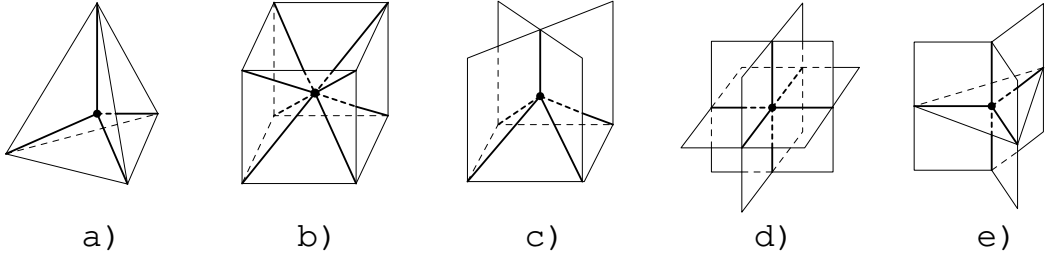}}
      \caption{5 possible fans of $(d-3)$-faces}
\end  {figure}

\begin{remark}
The significance of Theorem \ref{delaunay} becomes clear, as one considers its applications to known partial solutions of Voronoi's Conjecture \cite{vor1908}.
Several important examples are \cite{del29, zhi1929, ord2005}.
\end{remark}

In this paper we give a combinatorial proof of Theorem \ref{delaunay} using the notion of {\it jammed polyhedral fans}. Further, for every 
$(d-3)$-dimensional face $F$ we prove that the combinatorics of $\mathcal D(F)$ as an associated cell coincides with the combinatorics of 
$\conv \mathcal D(F)$. In particular, the proof implies that $\mathcal D(F)$ has affine dimension 3.

\begin{remark}
The term {\it jammed fan} (or {\it tight fan}) was first introduced by Andrey Gavrilyuk in private discussions during the Thematic
Program on Discrete Geometry at Fields Institute (2011). He argued that the notion of jammed fans is interesting on itself and certainly
deserves further investigation.
\end{remark}

Further, we explore how we can associate a lattice to a face $F$. The first option is to take the minimal sublattice of $\Lambda_P$ containing $\mathcal D(F)$.
This produces a 3-dimensional lattice we denote by $\Lambda(F)$. The second option is to consider the lattice
$$\Lambda_{\aff}(F) = \Lambda_P \cap \aff \mathcal D(F).$$
Obviously, $\Lambda(F) \subset \Lambda_{\aff}(F)$. We prove that these lattices coincide. In other terms, the index 
$$(\Lambda_{\aff}(F) : \Lambda(F))$$
is equal to 1. 

The author does not know, whether the analogous result on lattices holds for $(d-4)$-faces. For $(d-5)$-lattice the analogous statement is false since a 5-dimensional simplex of index 2 can be a Delaunay cell (claimed without any particular example by 
Voronoi~\cite[\S~67]{vor1908}, checked explicitly by Ryshkov~\cite{rys1970}).

\section{Jammed fans}

Let $\mathcal C$ be a $k$-dimensional complete polyhedral fan in $\mathbb R^k$. Consider an arbitrary face $C\in \mathcal C$ and let $\dim C = k-m$.

We can define the {\it fan of $C$} in the same way as fans of faces of parallelohedral tilings. Namely, let $L$ be an $m$-dimensional plane in
$\mathbb R^k$ intersecting $C$ transversally. In some neighborhood of $C\cap L$ the section of $\mathcal C$ by $L$ coincides with some complete
$m$-dimensional polyhedral fan, which is called the {\it fan of $C$}.

A cone $C\in \mathcal C$ is called {\it standard}, if its fan is centrally symmetric.

Let $C$ be a standard cone of a complete fan $\mathcal C$. We say that the cones $C_1, C_2 \in \mathcal C$ are symmetric
with respect to $C$ if the central symmetry at $C\cap L$ interchanges the $L$-sections of $C_1$ and $C_2$.

Obviously, for every standard cone $C\in \mathcal C$ all the faces $C'\in \mathcal C$ such that $C \subsetneq C'$ split into
pairs of centrally symmetric with respect to $C$.

A complete $k$-dimensional polyhedral fan $\mathcal C$ is called {\it jammed}, if for every two distinct $k$-dimensional cones 
$C_1, C_2 \in \mathcal C$ the cone $C = C_1\cap C_2$ is standard and $C_1$ and $C_2$ are centrally symmetric with respect to $C$.

Further, a face $F$ of $T(P)$ is called a {\it standard face} if $F$ is invariant under some central symmetry of $\Lambda_P$. It has been shown
in \cite{dol09} that the intersection of two distinct parallelohedra of $T(P)$ is a standard face or empty. Hence Proposition \ref{prop:1} follows.

\begin{proposition}\label{prop:1}
A fan of every face of a parallelohedral tiling is jammed.
\end{proposition}

To proceed with the proof of Theorem \ref{delaunay}, we will classify all jammed three-dimensional fans. Before we start with
the classification, we emphasize some properties of a jammed three-dimensional fan.

\begin{proposition}[Properties of three-dimensional jammed fan]\label{p2}

Let $\mathcal C$ be a three-dimensional jammed cone fan. Then

\begin{enumerate}
	\item [\rm 1.] Every ray $R\in \mathcal C$ has valence 3 or 4.
	\item [\rm 2.] For every two distinct three-dimensional cones $C_1, C_2 \in \mathcal C$ exactly one of the following statements is 
	correct:
	
\begin{itemize}
	\item [\rm (i)]   $C_1\cap C_2$ is a two-dimensional cone;
	\item [\rm (ii)]  $C_1\cap C_2$ is a ray of valence 4;
	\item [\rm (iii)] $\mathcal C$ is centrally symmetric, the central symmetry of $\mathcal C$ exchanges $C_1$ and $C_2$, and
	$C_1\cap C_2$ is the vertex of $\mathcal C$.
\end{itemize}
\end{enumerate}

\begin{proof}

It is obvious that a fan of every cone of a jammed cone fan is jammed. Consider a ray $R\in \mathcal C$. The fan of $R$ is
two-dimensional. But it is not hard to see that a jammed two-dimensional cone fan can consist only of 3 or 4 two-dimensional
cones. Therefore statement 1 follows.

To check statement 2 consider the intersection $C = C_1\cap C_2$. If $\dim C = 2$, then (i) is true. If $\dim C = 0$, then
(iii) is true by definition of a jammed fan. Finally, if $\dim C = 1$, then $C$ is a standard ray of $\mathcal C$. Therefore
$C$ has an even valence, so the only possible valence of $C$ is 4. Thus (ii) holds.

\end{proof}

\begin{remark}
Notice that the property 1 is essentially the Minkowski-Venkov condition on belts.
\end{remark}

\end{proposition}

\section{Asymmetric jammed fans}\label{sect:3}

Assume that $\mathcal C$ is an asymmetric three-dimensional jammed fan. Thus statement 2.(iii) of Proposition \ref{p2} does not
hold for any pair of distinct three-dimensional cones $C_1, C_2 \in \mathcal C$.

Let $a_3$ be the number of rays in $\mathcal C$ of valence 3, $a_4$ --- the number of rays of valence 4, $b$ --- the number of
two-dimensional cones and $c$ --- the number of three-dimensional cones.

Since every two-dimensional cone is incident to exactly 2 rays,
\begin{equation}\label{eq1}
b = \tfrac{3a_3}{2} + 2a_4.
\end{equation}

Consider a unit sphere $S^2$ centered at the vertex of $\mathcal C$. For every cone $C\in \mathcal C$ of dimension at least one
the intersection $C\cap S^2$ is a disk of dimension $\dim C - 1$. All these disks together can be treated as a representation of $S^2$ by a CW-complex
with Euler characteristic (see, for example, \cite[\S 13.5, Exercise 8]{ffu1989}) equal to~2. Thus
\begin{equation}\label{eq2}
a_3 + a_4 - b + c = 2.
\end{equation}

Finally, since every 2 three-dimensional cones are incident in one of the ways described in Proposition \ref{p2},
\begin{equation}\label{eq3}
\tfrac{c(c-1)}{2} = 2a_4 + b.
\end{equation}

After excluing $a_4$ and $b$ from the system consisting of equations (\ref{eq1}), (\ref{eq2}) and (\ref{eq3}), we get
$$\tfrac{c(c-1)}{2} = 4(c-2) - \tfrac{a_3}{2}.$$

Thus from the inequalities
$$\tfrac{c(c-1)}{2} \leq 4(c-2) \quad \text{and} \quad c\geq 4$$
follows that $c = 4, 5$ or $6$. Using (\ref{eq1})~-- (\ref{eq3}), we find that
$$(a_3, a_4, b, c) = (4,0,6,4),\; (4,1,8,5) \; \text{or} \; (2,3,9,6).$$

Enumerating all combinatorial types of complete three-dimensional polyhedral fans with no more than 6 rays, we find all the possible combinatorial
types of jammed fans. They are fans a), c) and e) in Figure 1. 

\section{Symmetric jammed fans}

For symmetric fans we use the same notation $a_3, a_4, b, c$ as in Section~\ref{sect:3}. Then the identities (\ref{eq1}) and (\ref{eq2})
hold because of the same argumentation as above. The identity (\ref{eq3}) is replaced by 
$$\tfrac{c(c-1)}{2} = 2a_4 + b + \tfrac{c}{2}, \eqno (3') $$
where the new summand $\tfrac{c}{2}$ appears, since in this case there are $\tfrac{c}{2}$ pairs of 3-dimensional cones of $\mathcal C$ with 0-dimensional
intersection.

As above, we exclude $a_4$ and $b$ from the system consisting of equations (\ref{eq1}), (\ref{eq2}) and ($3'$) and obtain
$$\tfrac{c(c-1)}{2} = 4(c-2) - \tfrac{a_3}{2}.$$

For symmetric complete polyhedral fans $c$ is even and at least 6. Thus the inequality
$$\tfrac{c(c-1)}{2} \leq 4(c-2)$$
implies $c=6$ or $c=8$. Consequently,
$$(a_3, a_4, b, c) = (0, 6, 12, 8) \; \text{or} \; (8, 0, 12, 6).$$
By enumerating all combinatorial types of complete centrally symmetric three-dimensional polyhedral fans with 6 or 8 rays, we conclude that in this case
only 2 combinatorial types of jammed fans are possible. They are shown in Figure 1, b) and d).

\section{Associated cells of $(d-3)$-dimensional faces}

Let $F$ be a face of $T(P)$. Suppose that the following two conditions hold.
\begin{enumerate}
	\item For every associated cell $\mathcal D(F')$ being a subcell of $\mathcal D(F)$, $\conv \mathcal D(F')$ is a face of $\conv \mathcal D(F)$.
	\item Conversely, for every face of $\conv \mathcal D(F)$ the vertex set of that face is a subcell of $\mathcal D(F)$.
\end{enumerate}
Then we say that $\mathcal D(F)$ {\it satisfies the duality conditions}.

In this section we prove the following result.

\begin{theorem}\label{th:5.1}

The associated cell of every $(d-3)$-dimensional face $F$ of a parallelohedral tiling satisfies the duality conditions.

\end{theorem}

\begin{proof}

First we show that 
\begin{equation}\label{eq:5.1}
\dim\aff \mathcal D(F) > 2.
\end{equation}

Assume that $\dim\aff \mathcal D(F) \leq 2$. Then, obviously,
$$\aff \mathcal D(F) = \aff \mathcal D(F'),$$
where $F'$ is an arbitrary $(d-2)$-dimensional superface of $F$.

Let $\pi_{F'}$ be a projection along $F'$ to a complementary 2-dimensional plane. It is known \cite{min1897, del29} that the set
$$\{\pi_{F'}(P'): P'\in T(P) \; \text{and} \; c(P') \in \aff \mathcal D(F')\}$$
is a tiling of a plane into hexagons if $F'$ is primitive or into parallelograms if $F'$ is standard.

In each case a) -- c), e) in Figure 1 the fan of $F$ contains a primitive $(d-2)$-face $F'$.
Then from the projection argument above follows that the set 
$$\{P'\in T(P) : c(P') \in \aff \mathcal D(F')\}$$
has no subset of 4 parallelohedra with a common point. But at least 4 parallelohedra share $F$, a contradiction.

In the case d) in Figure 1 let $F'$ be an arbitrary $(d-2)$-dimensional superface of $F$.
From the projection argument above we conclude that the set 
$$\{P'\in T(P) : c(P') \in \aff \mathcal D(F')\}$$
has no subset of 5 parallelohedra with a common point. But 8 parallelohedra share $F$, a contradiction.

Now let $G$ be a standard face of $T(P)$ and $P_{m_i} \in T(P)$ for $i = 1,2,3,4$. Suppose that
$$P_{m_1}\cap P_{m_2} = P_{m_3} \cap P_{m_4} = G.$$
Then, if $c(Q)$ denotes the center of symmetry of a polytope $Q$, 
\begin{equation}\label{eq:5.2}
c(P_{m_1})+c(P_{m_2}) = c(P_{m_3})+c(P_{m_4}),
\end{equation}
since both sides equal $2c(G)$.

Let $\mathcal D(F) = \{x_0, x_1, \ldots, x_n\}$, where $x_m = c(P_m)$. Then every equation of type (\ref{eq:5.2}) implies
\begin{equation}\label{eq:5.3}
(x_{m_1} - x_0) + (x_{m_2} - x_0) = (x_{m_3} - x_0) + (x_{m_4} - x_0).
\end{equation}

For each separate case a) -- e) in Figure 1 write down all possible equations of type (\ref{eq:5.3}). Solving the resulting system gives 3 vectors
$\mathbf e_1$, $\mathbf e_2$, $\mathbf e_3$ such that every vector $(x_m - x_0)$ ($m = 0, 1, \ldots, n$) is a linear combination of  
$\mathbf e_1$, $\mathbf e_2$ and $\mathbf e_3$ with particular real coefficients. Thus $\conv \mathcal D(F)$ is an image of a particular
3-dimensional polytope $D$ under some linear map. 

The inequality (\ref{eq:5.1}) shows that the map is not degenerate, so $\conv \mathcal D(F)$ is affinely (and then combinatorially)
equivalent to $D$. In each case a) -- e) the reader can use the system of equations of type (\ref{eq:5.3}) to specify $D$ and then check the duality conditions.

\end{proof}

\section{Lattices of $(d-3)$-dimensional faces}

\begin{theorem} 

For every $(d-3)$-dimensional face $F$ of a parallelohedral tiling holds
$$\Lambda_{\aff}(F) = \Lambda(F)).$$

\end{theorem}

\begin{proof}

Choose a point $x\in \relint F$ and fix a parallelohedron $P_0\in T(P)$ containing $F$ as a face. Let $P_1$ be some parallelohedron containing $F$.
The point $x$ belongs to $P_1$, so
$$x+c(P_0)-c(P_1) \in P_0.$$

Therefore if $P_0, P_1, \ldots, P_n$ are all the parallelohedra of $T(P)$ sharing $F$, then
$$\conv \{-c(P_i)+x+c(P_0) : i=0,1,\ldots,n \} \subset P_0.$$
Consequently, for every $P'\in T(P)$
$$(- \conv \mathcal D(F)) + x + c(P') = \conv \{x-c(P_i)+c(P') : i=0,1,\ldots,n \} \subset P'.$$

Restricting to parallelohedra satisfying $P' \in \Lambda_{\aff}(F)$, we conclude that the translates of the set $(- \conv \mathcal D(F)) + x $ 
by vectors of $\Lambda_{\aff}(F)$ do not overlap.

According to a well-known statement (see, for example, the paper \cite{dgr1962} by Danzer and Gr\"unbaum), the translates of the Minkowski symmetrization
$$\tfrac 12 (- \conv \mathcal D(F)) + \tfrac 12 \conv \mathcal D(F)$$
by vectors of $\Lambda_{\aff}(F)$ do not overlap as well.

This implies the inequality
\begin{equation}\label{eq:6.1}
V\left( \tfrac 12 (- \conv \mathcal D(F)) + \tfrac 12 \conv \mathcal D(F) \right) \leq V(\Lambda_{\aff}(F)),
\end{equation}
where the left part is the volume of a 3-dimensional set and the right part is the fundamental volume of the lattice.

Dividing the fundamental volume $V(\Lambda(F))$ by both parts of (\ref{eq:6.1}), we obtain
$$(\Lambda_{\aff}(F) : \Lambda(F)) \leq \frac{8V(\Lambda(F))}{V\left( (- \conv \mathcal D(F)) + \conv \mathcal D(F) \right)}.$$

For the cases b) -- e) in Figure 1 the last inequality immediately gives $(\Lambda_{\aff}(F) : \Lambda(F)) = 1$. For the case a) we have
$(\Lambda(F) : \Lambda_{\aff}(F)) \leq 2$.

Now suppose that $\conv \mathcal D(F)$ is a tetrahedron and $(\Lambda_{\aff}(F) : \Lambda(F)) = 2$. The latter condition means that
$$\Lambda_{\aff}(F) = \Lambda(F) \cup \left(\Lambda(F) + \tfrac 12 \mathbf t\right),$$
where $\mathbf t$ is a vector of $\Lambda(F)$. If $(\mathbf t_1 - \mathbf t_2)/2$ is a vector of $\Lambda(F)$, then the lattices
$$\Lambda(F) + \tfrac 12 \mathbf t_1 \quad \text{and} \quad \Lambda(F) + \tfrac 12 \mathbf t_2$$
coincide. Thus for fixed $\mathcal D(F)$ the lattice $\Lambda(F)$ is fixed and there are 7 candidates for $\Lambda_{\aff}(F)$.

Direct inspection shows that for each of the 7 possibilities each tetrahedron 
$$(- \conv \mathcal D(F)) + x + c(P')$$
has at least one edge, the midpoint of which belongs to an another tetrahedron of the same form. On the other hand, one can easily check that
all midpoint of eges of this tetrahedron lie in the interior of $P'$, a contradiction.

Thus in the case a) in Figure 1 $(\Lambda_{\aff}(F) : \Lambda(F)) = 1$ holds as well.

\end{proof}

\end{document}